\documentclass[12pt,reqno,a4paper]{amsart}

\usepackage{amssymb,amsthm}
\usepackage{amsfonts,cmtiup,comment,stmaryrd}

\usepackage{mathrsfs,cmtiup}

\makeatletter
\def\blfootnote{\xdef\@thefnmark{}\@footnotetext}
\makeatother

\newtheorem{theorem}{Theorem}[section]
\newtheorem{lemma}[theorem]{Lemma}
\newtheorem{proposition}[theorem]{Proposition}
\newtheorem{corollary}[theorem]{Corollary}

\theoremstyle{definition}

\newtheorem{remark}[theorem]{Remark}
\newtheorem*{definition*}{Definition}

\newcommand{\N}{\mathbb N}

\newcommand{\f}{\varphi}

\renewcommand{\geq}{\geqslant}
\renewcommand{\leq}{\leqslant}

\newcommand{\ed} {\end{document}}

\let\leq=\leqslant
\let\geq=\geqslant
\numberwithin{equation}{section}

\begin{document}
\title{Strong conciseness of Engel words\\ in profinite groups}

\author{E. I. Khukhro}
\address{Charlotte Scott Research Centre for Algebra, University of Lincoln, U.K.}
\email{khukhro@yahoo.co.uk}

\author{P. Shumyatsky}

\address{Department of Mathematics, University of Brasilia, DF~70910-900, Brazil}
\email{pavel@unb.br}


\begin{abstract}
A group word $w$ is said to be strongly concise in a class $\mathscr C$ of profinite groups if, for any group $G$ in $\mathscr C$, either $w$ takes at least continuum values in $G$ or  the verbal subgroup $w(G)$ is finite. It is conjectured that all words are strongly concise in the class of all profinite groups. Earlier  Detomi, Klopsch, and Shumyatsky proved this conjecture for multilinear commutator
words, as well as for some other particular words. They also proved that every group word is strongly concise in the class of nilpotent profinite groups. In the present paper we prove that for any $n$ the $n$-Engel word $[...[x,y],y],\dots y]$ (where $y$ is repeated $n$ times) is strongly concise in the class of finitely generated profinite groups.
\end{abstract}
\maketitle

\section{Introduction}
A group word $w$ is said to be \emph{strongly concise} in a class $\mathscr C$ of profinite groups if, for any group $G$ in $\mathscr C$, either $w$ takes at least continuum values in $G$ or  the verbal subgroup $w(G)$ is finite.
It is conjectured that all words are strongly concise in the class of all profinite groups.  Earlier Detomi, Klopsch, and Shumyatsky \cite{dks} proved this conjecture for multilinear commutator
words.  They also proved that every group word is strongly concise in the class of nilpotent profinite groups.
In the present paper we prove that for any $n$ the $n$-Engel word $[x,{}_ny]$  is strongly concise in the class of finitely generated profinite groups. Henceforth, we use the left-normed simple commutator notation
$[a_1,a_2,a_3,\dots ,a_r]:=[...[[a_1,a_2],a_3],\dots ,a_r]$ and the abbreviation $[a,{}_kb]:=[a,b,b,\dots, b]$ where $b$ is repeated $k$ times.

\begin{theorem}
\label{t-main}
For any $n$, the $n$-Engel word $[x,{}_ny]$ is strongly concise in the class of finitely generated profinite groups.
\end{theorem}

It is not yet clear if the $n$-Engel word is strongly concise in the class of all profinite groups.

By the Detomi--Klopsch--Shumyatsky theorem \cite[Theorem~1.2]{dks} on strong  conciseness of all words in the class of nilpotent profinite groups, Theorem~\ref{t-main} is an immediate consequence of the following result.

\begin{theorem}
\label{t1}
Let $n$ be a positive integer, and suppose that $G$ is a 
profinite group in which the word $[x,{}_ny]$  has strictly less than $2^{\aleph_0}$ values. Then $G$ has a finite normal subgroup $N$ such that $G/N$ is locally nilpotent.
\end{theorem}

Indeed, let $G$ be a  finitely generated profinite group in which   the word $[x,{}_ny]$ has strictly less than $2^{\aleph_0}$ values. Assuming that Theorem~\ref{t1} holds, we obtain a finite normal subgroup $N$ of $G$ such that $G/N$ is locally nilpotent. Since $G$ is finitely generated, $G/N$ is nilpotent. By \cite[Theorem~1.2]{dks} the verbal subgroup $\langle [x,{}_ny]\mid x,y\in G/N\rangle$ is finite, and since $N$ is finite, the verbal subgroup $\langle [x,{}_ny]\mid x,y\in G\rangle$ is also finite.

It is therefore the proof of Theorem~\ref{t1} that occupies the rest of the paper. Much of the technique used in this proof was developed by the authors earlier for studying profinite (and more generally compact) groups with finite or countable Engel sinks.  An \textit{Engel sink} of an element $g$ of a group $G$ is a set ${\mathscr E}(g)$ such that for every $x\in G$ all sufficiently long commutators $[x,g,g,\dots ,g]$ belong to ${\mathscr E}(g)$, that is, for every $x\in G$ there is a positive integer $n(x,g)$ such that
 $[x,{}_{n}g]\in {\mathscr E}(g)$ for all $n\geq n(x,g)$. (Thus, $g$ is an Engel element precisely when we can choose ${\mathscr E}(g)=\{ 1\}$, and $G$ is an Engel group when we can choose ${\mathscr E}(g)=\{ 1\}$ for all $g\in G$.) In \cite{khu-shu, khu-shu191} we considered  finite, profinite, and compact groups in which  every element has a finite or countable Engel sink  and proved the following theorem (the version with finite sinks was proved in \cite[Theorem~1.1]{khu-shu}).

\begin{theorem}[{\cite[Theorem~1.2]{khu-shu191}}]\label{t4.1}
If every element of a compact group $G$ has a countable Engel sink, then $G$ has a finite normal subgroup $N$ such that $G/N$ is locally nilpotent.
\end{theorem}

Here, ``countable'' stands for ``finite or  denumerable''. This result is a generalization of a theorem of  Wilson and Zelmanov \cite{wi-ze} saying that any Engel profinite group is locally nilpotent (which was later also extended to Engel compact  groups by Medvedev \cite{med}). 

It is easy to see that every element of a group $G$ in Theorem~\ref{t1}  has an Engel sink of cardinality strictly less than~$2^{\aleph_0}$. Therefore Theorem~\ref{t1} follows from Theorem~\ref{t4.1} under the Continuum Hypothesis. However, we aim at proving Theorem~\ref{t1} without assuming this additional axiom. The condition on the cardinalities being strictly less than~$2^{\aleph_0}$ rather than countable presents certain challenges. For example,  as shown by Ab\'ert \cite{abe}, the Baire Category Theorem (saying that if a compact Hausdorff group is a countable union of closed subsets, then one of these subsets has non-empty interior) cannot be proved in the version where the union of closed subsets  is taken over a set of cardinality less than  $2^{\aleph_0}$.
It remains an open question whether a stronger version of Theorem~\ref{t4.1} holds, with the hypothesis  weakened to every element having an Engel sink of cardinality less than  $2^{\aleph_0}$ (when the Continuum Hypothesis is not assumed).

The proof of Theorem~\ref{t1} is in many respects similar to the proof of Theorem~\ref{t4.1}  in \cite{khu-shu191}, albeit with certain modifications. First the case of pro-$p$ groups  is considered, where powerful Lie ring methods are applied including Zelmanov's theorem on Lie algebras satisfying a polynomial identity and generated by elements all of whose products are ad-nilpotent
\cite{ze92,ze95,ze17}.  Then the case of prosoluble groups is settled by using properties of Engel words and Engel sinks in coprime actions and a Hall--Higman--type theorem. The general case of profinite groups is dealt with by bounding the nonsoluble length of the group, which enables induction on this length. (We introduced the nonsoluble length in \cite{khu-shu131}, although bounds for nonsoluble length had been implicitly used in various earlier papers, for example, in the celebrated Hall--Higman paper \cite{ha-hi},
or in Wilson's  paper \cite{wil83}; more recently, bounds for the nonsoluble length were studied in connection with verbal subgroups in finite and profinite groups in \cite{dms1,flp, khu-shu132, 68, austral}.)

\section{Preliminaries}

In this section we recall some  notation and terminology and establish some important properties of Engel words and  Engel sinks in finite and profinite groups.

Our notation and terminology for profinite groups is standard; see, for example,  \cite{rib-zal, wil}.  A subgroup (topologically) generated by a subset $S$ is denoted by $\langle S\rangle$. Recall that centralizers are closed subgroups, while commutator subgroups $[B,A]=\langle [b,a]\mid b\in B,\;a\in A\rangle$ are the closures of the corresponding abstract commutator subgroups.

For a group $A$ acting by automorphisms on a group $B$ we use the usual notation for commutators $[b,a]=b^{-1}b^a$ and commutator subgroups $[B,A]=\langle [b,a]\mid b\in B,\;a\in A\rangle$, as well as for centralizers $C_B(A)=\{b\in B\mid b^a=b \text{ for all }a\in A\}$.

We record for convenience the following simple lemma.

\begin{lemma}[{see, for example,  \cite[Lemma~2.1]{khu-shu191}}]  \label{l-fng}
Suppose that $\varphi$ is a continuous automorphism of a compact group $G$ such that
$G=[G,\varphi ]$. If $N$ is a normal subgroup of $G$ contained in $C_G(\varphi )$, then $N\leq Z(G)$.
\end{lemma}

We denote by $\pi (k)$ the set of prime divisors of $k$, where $k$ may be a positive integer or a Steinitz number, and by $\pi (G)$ the set of prime divisors of the orders of elements of a (profinite) group $G$. Let $\sigma$ be a set of primes. An element $g$ of a group is  a $\sigma$-element if $\pi(|g|)\subseteq \sigma$, and a group $G$ is a $\sigma$-group if all of its elements are $\sigma$-elements. We denote by $\sigma'$ the complement of $\sigma$ in the set of all primes. When $\sigma=\{p\}$,  we write $p$-element, $p'$-element, etc.

Recall that a pro-$p$ group  is an inverse limit of finite $p$-groups, a pro-$\sigma $ group is an inverse limit of finite $\sigma$-groups, a pronilpotent group is an inverse limit of finite nilpotent groups, a prosoluble group is an inverse limit of finite soluble groups.

We denote by  $\gamma _{\infty}(G)=\bigcap _i\gamma _i(G)$ the intersection of the lower central series of a group~$G$. A profinite group $G$ is pronilpotent if and only if $\gamma _{\infty}(G)=1$.

Profinite groups have Sylow $p$-subgroups and satisfy analogues of the Sylow theorems.  Prosoluble groups satisfy analogues of the theorems of Hall and Chunikhin on Hall $\pi$-subgroups and Sylow bases. We refer the reader to the corresponding chapters in \cite[Ch.~2]{rib-zal} and \cite[Ch.~2]{wil}. We add a simple folklore lemma (see, for example, \cite[Lemma~2.2]{khu-shu191}).

\begin{lemma} \label{l-prosol-by-prosol}
A profinite group $G$ that is an extension of a prosoluble group $N$ by a prosoluble group $G/N$ is prosoluble.
\end{lemma}

Recall that the definition of an Engel sink was given in the introduction. Clearly, the intersection of two Engel sinks of a given element $g$ of a group $G$ is again an Engel sink of $g$, with the corresponding function $n(x,g)$ being the maximum of the two functions. Therefore, if $g$ has a \textit{finite} Engel sink, then $g$ has a unique smallest Engel sink. If $\mathscr E(g)$ is a smallest Engel sink of $g$, then the restriction of the mapping $x\mapsto [x,g]$ to $\mathscr E(g)$ must be surjective, which gives the following characterization.

\begin{lemma}[{\cite[Lemma~2.1]{khu-shu}}]\label{l-min}
If an element $g$ of a group $G$ has a finite Engel sink, then $g$ has a smallest Engel sink $\mathscr E (g)$ and for every $s\in \mathscr E (g)$ there is $k\in {\mathbb N}$ such that  $s=[s,{}_kg]$.
\end{lemma}

The following well-known fact is a straightforward consequence of the Baire Category Theorem (see \cite[Theorem~34]{kel}).

\begin{theorem}\label{bct}
If a compact Hausdorff group is a countable union of closed subsets, then one of these subsets has non-empty interior.
\end{theorem}

As shown by Ab\'ert \cite{abe}, an analogue of this theorem does not hold in the version where the union of closed subsets  is taken over a set of cardinality less than  $2^{\aleph_0}$ (without the Continuum Hypothesis). But a certain special case of such a generalization was recently obtained  in \cite{dks}.

\begin{proposition}[{\cite[Proposition 2.1]{dks}}] \label{pr-dks}
  Let $\f : X \to Y $ be a continuous map between non-empty profinite spaces such that the cardinality of the image $|\f (X)|$ is strictly smaller than $2^{\aleph_0}$. Then there exists a non-empty open subset $U$ of $ X$ such that the restriction $\f |_U$ is constant.
\end{proposition}

This proposition was used for deriving the following  fact.

\begin{lemma}[{\cite[Lemma~2.2]{dks}}] \label{l-c}
Let $G$ be a profinite group and let $x\in  G$. If the conjugacy class $\{x^g \mid  g\in  G\}$
contains less than  $2^{\aleph_0}$ elements, then it is finite.
\end{lemma}

We now use Proposition~\ref{pr-dks} for proving two important technical lemmas about values of $n$-Engel words in profinite groups, which will be crucial in the proof of Theorem~\ref{t1}.

\begin{lemma}\label{l1}
Suppose that for an element $g$ of a profinite
group $G$ the cardinality of the set of $n$-Engel words $\{[h,{}_ng]\mid h\in G\}$
is strictly smaller than $2^{\aleph_0}$.
Then there is an element $s\in G$ and a coset $Nb$ of an open normal subgroup $N$ such that
$$
[xb,{}_ng]=s \qquad \text{for all}\quad x\in N.
$$
\end{lemma}

\begin{proof}
The mapping
$$
\f: h\to [h,{}_ng],\qquad h\in G,
$$
is continuous. Hence the result follows from Proposition~\ref{pr-dks}.
\end{proof}

\begin{lemma}\label{l2}
Suppose that for an element $g$ of a profinite
group $G$ the cardinality of the set of $n$-Engel words $\{[h,{}_ng]\mid x\in G\}$
is strictly smaller than $2^{\aleph_0}$.
 Then there is a positive integer $k$ and a coset $Nb$ of an open normal subgroup $N$  such that
$$
[[xb,{}_ng],g^k]=1 \qquad \text{for all}\quad x\in N.
$$
\end{lemma}

\begin{proof}
By Lemma~\ref{l1}, there is an element $s\in G$ and a coset $Nb$ of an open normal subgroup $N$ such that
$$
[xb,{}_ng]=s\qquad \text{for all}\quad x\in N.
$$
Since $G/N$ is a finite group, the coset $Nb$ is invariant under conjugation by some power $g^k$. Then
$$
\begin{aligned}
s^{g^k}&=[b,{}_ng]^{g^{k}}=[b^{g^{k}},{}_ng]\\&=[xb,{}_ng]\quad \text{for some } x\in N\\
&=s.
\end{aligned}
$$
In other words, $g^{k}$ commutes with $s$, so that
\begin{equation*}
[[xb,{}_ng],g^{k}]=[s,g^k]=1\qquad \text{for all}\quad x\in N.\tag*{\qedhere}
\end{equation*}
\end{proof}

\begin{remark} The condition that the word $[x,{}_ny]$ has strictly less than $2^{\aleph_0}$ values in a group $G$ is inherited by every section of $G$, and we shall freely use this property without special references.
\end{remark}

\section{Pronilpotent groups}\label{s-pron}

When $G$ is a pro-$p$ group, or more generally a pronilpotent group, the conclusion of Theorem~\ref{t1} is equivalent to $G$ being locally nilpotent, and this is what we prove in this section.

\begin{theorem}
\label{t2}
Suppose that $G$ is a 
pronilpotent group in which the word $[x,{}_ny]$ has strictly less than $2^{\aleph_0}$ values. Then $G$ is locally nilpotent.
\end{theorem}

The bulk of the proof is about the case where $G$ is a pro-$p$ group. First we remind the reader of important Lie ring methods in the theory of pro-$p$ groups.

For a prime number $p $, the \textit{Zassenhaus $p $-filtration} of a group $G$ (also called the \textit{$p $-dimension series}) is defined by
$$
G_i=\langle g^{p ^k}\mid g\in \gamma _j(G),\;\, jp ^k\geqslant i\rangle \qquad\text{for}\quad i=1,2,\dots
$$
This is indeed a \textit{filtration} (or an \textit{$N$-series}, or a \textit{strongly central series}) in the sense that
\begin{equation}\label{e-fil}
[G_i,G_j] \leqslant G_{i+j}\qquad \text{for all}\quad i, j.
 \end{equation}
 Then the Lie ring $D_p (G)$ is defined with the additive group
$$
D_p(G)=\bigoplus _{i}G_i/G_{i+1},
$$
where the  factors $Q_i=G_i/G_{i+1}$ are additively written. The Lie product is defined on homogeneous elements $xG_{i+1}\in Q_i$, $yG_{j+1}\in Q_j$ via the group commutators by
$$
[xG_{i+1},\, yG_{j+1}] = [x, y]G_{i+j+1}\in Q_{i+j}
$$
and extended to arbitrary elements of $D_p(G)$ by linearity. Condition~\eqref{e-fil} ensures that this  product is well-defined, and group commutator identities imply that $D_p(G)$ with these operations is a Lie ring. Since all the factors $G_i/G_{i+1}$ have prime exponent~$p $, we can view $D_p(G)$ as a Lie algebra over the field of $p $ elements $\mathbb{F}_p $. We denote  by $L_p (G)$ the subalgebra generated by the first factor $G/G_2$. (Sometimes, the notation $L_p (G)$ is used for $D_p (G)$.)

A group $G$ is said to satisfy a \textit{coset identity} if there is a group word $w(x_1,\dots ,x_m)$ and cosets $a_1H,\dots ,a_mH$ of a subgroup $H\leqslant G$  such that $w(a_1h_1,\dots ,a_mh_m)=1$ for any $h_1,\dots ,h_m\in H$. We shall use the following  result of Wilson and Zelmanov \cite{wi-ze} about coset identities.

\begin{theorem}[{Wilson and Zelmanov \cite[Theorem~1]{wi-ze}}]\label{t-coset}
If a group $G$ satisfies a coset identity on cosets of a subgroup of finite index, then for every prime $p$ the Lie algebra $L_p (G)$  satisfies a polynomial identity.
\end{theorem}

 Theorem~\ref{t-coset} was used in the proof of the above-mentioned theorem on profinite Engel groups, which we state here for convenience.

\begin{theorem}[{Wilson and Zelmanov \cite[Theorem~5]{wi-ze}}]\label{t-wz}
Every profinite Engel group is locally nilpotent.
\end{theorem}

The proof of Theorem~\ref{t-wz} was based on the following deep result of Zelmanov \cite{ze92,ze95,ze17}, which is also used in our paper.

\begin{theorem}[{Zelmanov \cite{ze92,ze95,ze17}}]\label{tz}
 Let $L$ be a Lie algebra over a field and suppose that $L$ satisfies
a polynomial identity. If $L$ can be generated by a finite set $X$ such that every
commutator in elements of $X$ is ad-nilpotent, then $L$ is nilpotent.
\end{theorem}

We now consider the case of pro-$p$ groups.

\begin{proposition}
\label{pr-pro-p}
Suppose that $P$ is a finitely generated pro-$p$ group in which the word $[x,{}_ny]$ has strictly less than $2^{\aleph_0}$ values.  Then $P$ is nilpotent.
\end{proposition}

\begin{proof}
We shall first prove that the Lie algebra $L_p(P)$ is nilpotent, using Theorem~\ref{tz}. The next lemma confirms that the hypotheses in Theorem~\ref{tz} are satisfied.

\begin{lemma}\label{l-ad}
The Lie algebra $L_p(P)$ satisfies a polynomial identity and is generated by finitely many elements all commutators in which are ad-nilpotent.
\end{lemma}

\begin{proof}
The proof of this lemma is obtained by repeating word-for-word the proofs of Lemmas~3.6 and 3.7 in \cite{khu-shu191}, where Lemma~2.7 in \cite{khu-shu191} is replaced with Lemma~\ref{l2} in the present paper, which provides exactly the same result as  in \cite{khu-shu191} under the hypotheses of Proposition~\ref{pr-pro-p}.
\end{proof}

We now resume the proof of Proposition~\ref{pr-pro-p}. Lemma~\ref{l-ad} together with Theorem~\ref{tz} show that  $L_p(P)$ is nilpotent.
The nilpotency of the Lie algebra $L_p (P)$  of the finitely generated pro-$p $ group $P$  implies that $P$ is a $p $-adic analytic group. This result goes back to Lazard~\cite{laz}; see also \cite[Corollary~D]{sha}.
Furthermore, by a theorem of Breuillard and Gelander \cite[Theorem~8.3]{br-ge}, a  $p $-adic analytic group satisfying a coset identity on cosets of a subgroup of finite index is soluble.

Thus, $P$ is soluble, and we prove that $P$ is nilpotent by induction on the derived length of $P$.  By induction hypothesis, $P$ has an abelian normal subgroup $U$ such that $P/U$ is  nilpotent. We aim to show that $P$ is an Engel group. Since $P/U$ is nilpotent, it is sufficient to show that every element $a\in P$ is an Engel element in the product $U\langle a\rangle$.

Applying Lemma~\ref{l1} to $U\langle a\rangle$ we obtain a coset $Nb$ of an open normal subgroup $N$ of $U\langle a\rangle$ and an element $s\in U$ such that
$$
[xb,{}_na]=s\qquad \text{for all}\;\,x\in N.
$$
  Since $[a^iux,a]=[ux,a]$ for any $u\in U$, $x\in N$, we can assume that $b\in U$. Then for any $m\in U\cap N$ we have
$$
s=[mb,{}_na]=[m,{}_na]\cdot [b,{}_na]=[m,{}_na]\cdot s,
$$
since $U$ is abelian. Hence, $[m,{}_na]=1$ for any $m\in U\cap N$. Since $U\cap N$ has finite index in $U$ and $U\langle a\rangle$ is a pro-$p$ group, it follows that $a$ is an Engel element of  $U\langle a\rangle$.

Thus, $P$ is an Engel group and therefore, being a finitely generated pro-$p$ group,  $P$ is nilpotent by
Theorem~\ref{t-wz}.
\end{proof}

\begin{proof}[Proof of Theorem~\ref{t2}] By
Theorem~\ref{t-wz},  it is sufficient to prove that $G$ is an Engel group.
For each prime~$p$, let $G_p$ denote the Sylow $p$-subgroup of $G$, so that $G$ is a Cartesian product of the~$G_p$, since $G$ is pronilpotent. Given any two elements $a,g\in G$, we write $g=\prod _pg_p$ and $a=\prod _pa_p$,  where $a_p,g_p\in G_p$. Clearly, $[g_q,a_p]=1$ for $q\ne p$.

By Lemma~\ref{l2}, for the element $a\in G$ there is a positive integer $k$ and a coset $Nb$ of an open normal subgroup $N$  such that
\begin{equation}\label{e-eng2}
[[xb,{}_na],a^{k}]=1\qquad \text{for all}\quad x\in N.
\end{equation}
Let $l$ be the (finite) index of $N$ in $G$. Then $N$ contains all Sylow $q$-subgroups of $G$ for $q\not\in \pi (l)$. Hence we can choose $b$ to be a $\pi(l)$-element. Let $\pi=\pi (l)\cup \pi (k)$; note that $\pi$ is a finite set of primes.

We claim that
$$
[g_q,{}_{n+1}a_q]=1\qquad \text{for }q\not\in \pi.
$$
 Indeed, since $b$ commutes with elements of $G_q$ and $G_q\leq N$,  by \eqref{e-eng2} we have
\begin{equation}\label{eq-engq2}
\begin{aligned}
 1=[g_qb,{}_na],a^{k}]
 & = [[g_q,{}_na],a^{k}]\cdot  [[b,{}_na],a^{k}]\\
 & =[[g_q,{}_na],a^{k}]\\
 &= [[g_q,{}_na_q],a_q^{k}].
 \end{aligned}
\end{equation}
 Thus, $a_q^{k}$ centralizes $[g_q,{}_na_q]$. Since $k$ is coprime to $q$, we have $ \langle a_q^{k}\rangle =\langle a_q\rangle$. Therefore \eqref{eq-engq2} implies that $[[g_q,{}_na_q],a_q]=1$, as claimed.

 For every prime $p$ the group $G_p$ is locally nilpotent by Proposition~\ref{pr-pro-p},  so there is $k_p$ such that $[g_p,{}_{k_p}a_p]=1$.  Now for $m=\max\{n+1, \max_{p\in \pi} \{k_p\}\}$ we have $[g_p,{}_{m}a_p]=1$ for all $p$, which means that  $[g,{}_{m}a]=1$.
 Thus, $G$ is an Engel group and therefore it is locally nilpotent by
 Theorem~\ref{t-wz}.
\end{proof}

\section{Coprime actions}

In this section,  first we list several profinite  analogues of the properties of coprime automorphisms of finite groups. Then we prove two lemmas on coprime automorphisms in relation to Engel sinks and  values of Engel words.

 If $\varphi$  is an automorphism of a finite group $H$ of coprime order, that is, such that $(|\varphi |,|H|)=1$, then  we say for brevity that $\varphi$ is a coprime automorphism of~$H$. This definition is extended to profinite groups as follows.
We say that $\varphi$ is a \textit{coprime automorphism}  of a profinite group $H$  meaning that a procyclic group $\langle\varphi\rangle$ faithfully acts on $H$ by continuous automorphisms   and $\pi (\langle \varphi\rangle)\cap \pi (H)=\varnothing$. Since the semidirect product $H\langle \varphi\rangle$ is also a profinite group, $\varphi$ is a coprime automorphism of $H$
 if and only if for every open normal $\varphi$-invariant subgroup $N$ of $H$ the automorphism (of finite order) induced by $\varphi$ on $H/N$ is a coprime automorphism.
The following folklore lemma follows from the Sylow theory for profinite groups and an analogue of the Schur--Zassenhaus theorem (see, for example, {\cite[Lemma~4.1]{khu-shu191}).

\begin{lemma}\label{l-inv}
If $\varphi$ is a coprime automorphism of a profinite group $G$, then for every prime $q\in \pi (G)$ there is a $\varphi$-invariant Sylow $q$-subgroup of $G$. If $G$ is in addition prosoluble, then for every subset $\sigma\subseteq \pi (G)$ there is a $\varphi$-invariant Hall $\sigma$-subgroup of~$G$.
\end{lemma}

The following lemma is a special case of \cite[Proposition~2.3.16]{rib-zal}.

\begin{lemma}[{\cite[Lemma~4.2]{khu-shu191}}] \label{l-cover}
If $\varphi $ is a coprime automorphism of a profinite group $G$ and $N$ is a $\varphi $-invariant closed normal subgroup of $G$, then every fixed point of $\varphi$ in $G/N$ is an image of a fixed point of $\varphi$ in $G$, that is, $C_{G/N}(\varphi)=C(\varphi )N/N$.
\end{lemma}

As a consequence of Lemma~\ref{l-cover}, we also have the following folklore lemma.

\begin{lemma}[{\cite[Lemma~4.3]{khu-shu191}}]  \label{l-gff}
If $\varphi $ is a coprime automorphism of a profinite group $G$, then
$[[G,\varphi],\varphi]=[G,\varphi]$.
\end{lemma}

The following lemma is a consequence of \cite[Lemma~2.4]{rod-shu}.

 \begin{lemma}[{\cite[Lemma~4.6]{khu-shu191}}] \label{l-copr1}
 Let $\varphi$ be a coprime automorphism of a pronilpotent
group~$G$. Then the restriction of the mapping
$$
\theta: x\mapsto [x,\varphi]
$$
to the set $K=\{ [g,\varphi]\mid g\in G\}$ is injective.
 \end{lemma}

We now prove two lemmas where the condition of the word $[x, \,{}_ny]$ having less than $2^{\aleph_0}$ values appears. 

 \begin{lemma}\label{l-copr2}
 Let $\varphi$ be a coprime automorphism of a pronilpotent
group~$G$. Suppose that, for some $n\in \N$, the set $\mathscr E_{G,n}(\varphi )=\{[g, \,{}_n\varphi]\mid g\in G\}$ has less than $2^{\aleph_0}$ elements. Then the set $K=\{ [g,\varphi]\mid g\in G\}$ is a finite smallest Engel sink of $\varphi$ in the semidirect product $G\langle \varphi\rangle$.
 \end{lemma}

\begin{proof}
Since the mapping $\theta: x\mapsto [x,\varphi]$ is injective on the set $K$
by Lemma~\ref{l-copr1}, the sets
$\mathscr E_{G,n}(\varphi )$ and $K$ have the same cardinality, which is less than $2^{\aleph_0}$ by hypothesis. The set $K=\{ [g,\varphi]\mid g\in G\}$ is in a one-to-one correspondence with the set of (say, right) cosets of the centralizer $C_G(\varphi )$. But this set of cosets cannot be infinite of cardinality less than $2^{\aleph_0}$ by Lemma~\ref{l-c}. Therefore it is finite, and so is the set $K$.

 The mapping $[g,\varphi ]\mapsto [g,\varphi,\varphi]$ is injective on $K$ by Lemma~\ref{l-copr1}, and therefore it is also surjective, since $K$ is finite. Hence this set is a smallest finite Engel sink of~$\varphi$.
 \end{proof}

 \begin{lemma} \label{l-copr3}
 Let $\varphi$ be a coprime automorphism of a pronilpotent
group~$G$. Suppose that, for some $n\in \N$, the word $[x,{}_ny]$ has strictly less than $2^{\aleph_0}$ values in the semidirect product $G\langle \varphi\rangle$.
Then $\gamma _{\infty} (G\langle \varphi\rangle)$ is finite  and  $\gamma _{\infty}(G\langle \varphi\rangle)= [G,\varphi]$.
 \end{lemma}

 \begin{proof}
The group $G$ is locally nilpotent by Theorem~\ref{t2}. By Lemma~\ref{l-copr2}, the set $K=\{ [g,\varphi]\mid g\in G\}$ is finite. Therefore the commutator subgroup $[G,\varphi ]=\langle K\rangle$  is nilpotent.
By Lemma~\ref{l-gff},
\begin{equation}\label{e-ff}
[[G,\varphi],\varphi]=[G,\varphi].
\end{equation}

Let $V$ be the quotient of  $[G,\varphi ]$ by its derived subgroup. For any $u,v\in V$ we have $[uv,\varphi ]=[u,\varphi][v,\varphi]$, since $V$ is abelian, and $[V,\varphi ]=V$ by \eqref{e-ff}. Hence $V$ consists of the images of elements of $K$, and therefore is finite. Then the nilpotent group $[G,\varphi]$ is also finite (see, for example, \cite[5.2.6]{rob}).

The quotient $G\langle \varphi\rangle/[G,\varphi]$ is obviously the direct product of the images of $G$ and $\langle \varphi\rangle$ and therefore is pronilpotent. Hence,  $\gamma _{\infty}(G\langle \varphi\rangle)\leq [G,\varphi]$, so $\gamma _{\infty}(G\langle \varphi\rangle)$ is finite. Since the set of commutators $\{ [g,\varphi]\mid g\in G\}$ is the smallest Engel sink of $\varphi$ by Lemma~\ref{l-copr2}, it follows that  $\gamma _{\infty}(G\langle \varphi\rangle)= [G,\varphi]$.
 \end{proof}

\section{Prosoluble groups}

In this section we prove Theorem~\ref{t1} for prosoluble groups.
First we consider the case of prosoluble groups of finite Fitting height. Recall that by Theorem~\ref{t2} a pronilpotent group in which the word $[x,{}_ny]$ has strictly less than $2^{\aleph_0}$ values is locally nilpotent. Therefore, if $G$ is a profinite group in which the word $[x,{}_ny]$ has strictly less than $2^{\aleph_0}$ values, then the largest pronilpotent normal subgroup $F(G)$  is also the largest locally nilpotent normal subgroup, and we call it the Fitting subgroup of $G$. Then further terms of the Fitting series are defined as usual by induction: $F_1(G)=F(G)$ and $F_{i+1}(G)$ is the inverse image of $F(G/F_i(G))$. A group has finite Fitting height if $F_k(G)=G$ for some $k\in {\mathbb N}$.

 \begin{proposition}\label{pr-height}
 Let $G$ be a  prosoluble group in which the word $[x,{}_ny]$ has strictly less than $2^{\aleph_0}$ values. If $G$ has finite Fitting height, then $\gamma _{\infty} (G)$ is finite.
 \end{proposition}

\begin{proof}  It is sufficient to prove the result for the case of Fitting height 2. Then the general case will follow by induction on the Fitting height $k$ of $G$. Indeed, then $\gamma _{\infty}(G/ \gamma _{\infty}(F_{k-1}(G)))$ is finite, while $\gamma _{\infty}(F_{k-1}(G))$ is finite by the induction hypothesis, and as a result, $\gamma _{\infty}(G)$ is finite.

Thus, we assume that $G=F_2(G)$. By Theorem~\ref{t4.1}, it is sufficient to show that every element  $a\in G$ has a finite Engel sink.  Since $G/F(G)$ is locally nilpotent, an Engel sink of $a$ in $F(G)\langle a\rangle$ is also an Engel sink of $a$ in $G$.

For a prime $p$, let $P$ be a Sylow $p$-subgroup of $F(G)$, and write $a=a_pa_{p'}$, where $a_p$ is a $p$-element, $a_{p'}$ is a $p'$-element, and $[a_p,a_{p'}]=1$. Then $P\langle a_p\rangle$ is a normal Sylow $p$-subgroup of  $P\langle a\rangle$, on which $a_{p'}$ induces by conjugation a coprime automorphism. By Lemma~\ref{l-copr3} the subgroup $\gamma _{\infty}(P\langle a\rangle)=[P,a_{p'}]$ is finite. Since the pronilpotent group $P\langle a\rangle/\gamma _{\infty}(P\langle a\rangle)$ is locally nilpotent by Theorem~\ref{t2},  we can choose a finite smallest Engel sink $\mathscr E_p(a)\subseteq \gamma _{\infty}(P\langle a\rangle)$ of $a$ in $P\langle a\rangle$.

Note that
\begin{equation}\label{e-equiv}
   \text{if}\quad \mathscr E_p(a)=\{ 1\}, \quad\text{then}\quad \gamma _{\infty}(P\langle a\rangle)=1.
\end{equation}
Indeed, if  $\mathscr E_p(a)=\{ 1\}$, then, in particular, the image $\bar a$ of $a$ in $\langle a\rangle/C_{\langle a\rangle}([P,a_{p'}])$ is an Engel element of the finite group $[P,a_{p'}]\langle \bar a\rangle$ and therefore $\bar a$ is contained in its Fitting subgroup by  Baer's theorem \cite[Satz~III.6.15]{hup}. Then
$$
\gamma _{\infty}(P\langle a\rangle)=[P,a_{p'}]=[[P,a_{p'}],a_{p'}]=[[P,a_{p'}],\bar a_{p'}]=1.
$$

By Lemma~\ref{l-min}, for every $s\in \mathscr E_p(a)$ we have $s=[s,{}_ka]$ for some $k\in {\mathbb N}$, and then also
\begin{equation}\label{e-cycl}
s=[s,{}_{kl}a]\quad \text{for any}\;\, l\in {\mathbb N}.
\end{equation}

We claim that $\mathscr E_p(a)=\{1\}$ for all but finitely many primes $p$. Suppose the opposite, and $\mathscr E_{p_i}(a)\ne \{1\}$ for each prime $p_i$ in an infinite set of primes~$\pi$. Choose a nontrivial element $s_{p_i}\in \mathscr E_{p_i}(a)$ for every $p_i\in \pi$. For any subset $\sigma\subseteq \pi$, consider the (infinite) product
$$
s_{\sigma}=\prod _{p_j\in \sigma} s_{p_j}.
$$
Note that the elements $s_{p_j}$ commute with one another  belonging to different normal Sylow subgroups of $F(G)$.  If  $\mathscr E(a)$ is any Engel sink of $a$ in $G$, then for some $k\in {\mathbb N}$ the commutator $[s_{\sigma},{}_ka]$ belongs to $\mathscr E(a)$. Because of the properties \eqref{e-cycl}, all the components of $[s_{\sigma},{}_ka]$ in the Sylow $p_j$-subgroups of $F(G)$ for $p_j\in \sigma$ are non-trivial, while all the other components in Sylow $q$-subgroups for $q\not\in \sigma$ are trivial by construction. Therefore for different subsets $\sigma\subseteq \pi$ we thus obtain different elements of  $\mathscr E(a)$. The infinite set $\pi$ has $2^{\aleph_0}$ different subsets, whence $\mathscr E(a)$ has  cardinality at least $2^{\aleph_0}$. But we can choose
$$\mathscr E(a)=\bigcup_{i\in \N}\{[g,{}_{n+i}a]\mid g\in G\},$$
which is a countable union of sets each having cardinality less than  $2^{\aleph_0}$ by hypothesis. This Engel sink $\mathscr E(a)$ therefore also has cardinality less than  $2^{\aleph_0}$, a contradiction.

Thus, for all but finitely many primes $p$ we have $\mathscr E_p(a)=\{1\}$, which is the same as $\gamma _{\infty}(P\langle a\rangle)=1$ by \eqref{e-equiv}. Therefore the subgroup
$$
\gamma _{\infty}(F(G)\langle a\rangle)=\prod_p\gamma _{\infty}(P\langle a\rangle)
$$
 is finite. The quotient $F(G)\langle a\rangle/\gamma _{\infty}(F(G)\langle a\rangle)$ is pronilpotent and therefore locally nilpotent by Theorem~\ref{t2}. Hence we can choose a finite Engel sink for $a$ in $G$ as a subset of $\gamma _{\infty}(F(G)\langle a\rangle)$.

Thus, every element of $G$ has a finite Engel sink,  and therefore $\gamma _{\infty}(G)$ is finite by Theorem~\ref{t4.1}.
\end{proof}

\begin{lemma}\label{l-finp}
Let $\varphi $ be a coprime automorphism
of a prosoluble group $G$ such that the set of primes $\pi (G)$ is finite. If the word $[x,{}_ny]$ has strictly less than $2^{\aleph_0}$ values in the semidirect product $G\langle \varphi\rangle$, then the subgroup $[G,\varphi ]$ is finite.
\end{lemma}

\begin{proof}
By Lemma~\ref{l-gff} we can assume that $G=[G,\varphi ]$. For every prime $q\in \pi (G)$ there is a $\varphi$-invariant Sylow $q$-subgroup $G_q$ of $G$ by Lemma~\ref{l-inv}. By Lemma~\ref{l-copr2} the set $\{[g,\varphi]\mid g\in G_q\}$ is finite. Since $\pi (G)$ is finite, there is an open normal subgroup $N$ of $G$ that intersects trivially with every set $\{[g,\varphi]\mid g\in G_q\}$, which implies that $\varphi$ centralizes every Sylow $q$-subgroup $N\cap G_q$ and therefore $[N,\varphi]=1$. Since $N$ is normal and $G=[G,\varphi ]$, we obtain $[N,G]=1$ by Lemma~\ref{l-fng}. Thus, $G/Z(G)$ is finite and, in particular, the Fitting height of $G$ is finite. Then $\gamma _{\infty} (G\langle \varphi\rangle)$ is finite by Proposition~\ref{pr-height}, and therefore $[G,\varphi]$ is also finite, since $[G,\varphi ]\leq  \gamma _{\infty} (G\langle \varphi\rangle)$ by Lemma~\ref{l-copr3}  applied to $G\langle \varphi\rangle/ \gamma _{\infty} (G\langle \varphi\rangle)$.
\end{proof}

\begin{proposition}\label{pr-f}
If the word $[x,{}_ny]$ has strictly less than $2^{\aleph_0}$ values in a prosoluble group~$G$, then $F(G)\ne 1$.
\end{proposition}

\begin{proof}
The proof of this proposition is obtained by  repeating word-for-word the proof of Proposition~5.4 in \cite{khu-shu191}, where Proposition~5.1 in \cite{khu-shu191} is replaced with Proposition~\ref{pr-height} in the present paper,  Lemma~4.8  in \cite{khu-shu191} is replaced with Lemma~\ref{l-copr3} in the present paper, and Theorem~3.1 in \cite{khu-shu191} is replaced with Theorem~\ref{t2} in the present paper. The corresponding lemmas, proposition, and theorem provide exactly the same results as in \cite{khu-shu191} under the hypotheses of Proposition~\ref{pr-f}.
\end{proof}

We are now ready to prove the main result of this section.

\begin{theorem}\label{t3}
Suppose that $G$ is a prosoluble group in which  the word $[x,{}_ny]$ has strictly less than $2^{\aleph_0}$ values. Then $G$ has a finite normal subgroup $N$ such that $G/N$ is locally nilpotent.
\end{theorem}

\begin{proof}
By Theorem~\ref{t2} it is sufficient to prove that $\gamma _{\infty}(G)$ is finite. 
By Proposition~\ref{pr-height} we obtain that $\gamma _{\infty}(F_2(G))$
is finite and the quotient $F_2(G)/\gamma _{\infty}(F_2(G))$ is locally nilpotent by Theorem~\ref{t2}. Then the subgroup $C=C_{F_2(G)}(\gamma _{\infty}(F_2(G)))$ has finite index in $F_2(G)$ and is locally nilpotent. Indeed, for any finite subset $S\subseteq C_{F_2(G)}(\gamma _{\infty}(F_2(G)))$ we have $\gamma _k(\langle S\rangle)\leq \gamma _{\infty}(F_2(G))$ for some $k$, and then
$$
\gamma _{k+1}(\langle S\rangle)=[\gamma _k(\langle S\rangle), \langle S\rangle]\leq [\gamma _{\infty}(F_2(G)), C]=1.
$$
As a normal locally nilpotent subgroup, $C$ is contained in $F(G)$. Hence, $F_2(G)/F(G)$ is finite.

We claim that the quotient $G/F(G)$ is finite. Let the bar denote the images in $\bar G=G/F(G)$. Then $F(\bar G)=\overline{F_2(G)}$ is finite by the above. There is an open normal subgroup $N$ of $\bar G$ such that $N\cap  F(\bar G)=1$. If $N\ne 1$, then $F(N)\ne 1$ by Proposition~\ref{pr-f}. But $F(N)\leq N\cap F(\bar G)=1$; hence we must have $N=1$, so $\bar G$ is finite.

Thus, $G/F(G)$ is finite, and therefore $G$ has finite Fitting height. By Proposition~\ref{pr-height} we obtain that $\gamma _{\infty}(G)$ is finite, as required.
\end{proof}

A proof of the next result can be obtained as in \cite[Corollary~5.6]{khu-shu191} with only obvious modifications, so we omit details.

\begin{corollary}\label{c-virt}
Suppose that $G$ is a virtually prosoluble group in which the word $[x,{}_ny]$ has strictly less than $2^{\aleph_0}$ values. Then $G$ has a finite normal subgroup $N$ such that $G/N$ is locally nilpotent.
\end{corollary}

\section{Profinite groups}

The proof of Theorem~\ref{t1} uses induction on  so-called nonprosoluble length.
We recall the relevant definitions.  The  \textit{nonsoluble length} $\lambda (H)$  of a finite group $H$ is defined as the minimum number of nonsoluble factors in a normal series in which every  factor  either is soluble or is a direct product of non-abelian simple groups. (In particular, the group is soluble if and only if its nonsoluble length is $0$.) Clearly, every finite group has a normal series with these properties, and therefore its nonsoluble length is well defined.  It is easy to see that the nonsoluble length $\lambda (H)$ is equal to the least positive integer $l$ such that there is a series of characteristic subgroups
\begin{equation*}
1=L_0\leqslant R_0 <  L_1\leqslant R_1<  \dots \leqslant R_{l}=H
\end{equation*}
in which each quotient $L_i/R_{i-1}$ is a (nontrivial) direct product of non-abelian simple groups, and each quotient $R_i/L_{i}$ is soluble (possibly trivial).

We shall use the following result of Wilson \cite{wil83}, which we state in the special case of $p=2$ using the terminology of nonsoluble length.

\begin{theorem}[{see \cite[Theorem~2*]{wil83}}]\label{t-wil83}
Let $K$ be a normal subgroup of a finite group~$G$. If a Sylow $2$-subgroup $Q$ of $K$ has a coset $tQ$ of exponent dividing $2^k$, then the nonsoluble length of $K$ is at most $k$.
\end{theorem}

It is natural to say that a profinite group $G$ has finite \textit{nonprosoluble length} at most $l$ if $G$ has a normal series \begin{equation*}
1=L_0\leqslant R_0 <  L_1\leqslant R_1<  \dots \leqslant R_{l}=G
\end{equation*}
in which each quotient $L_i/R_{i-1}$ is a (nontrivial) Cartesian product of non-abelian finite simple groups, and each quotient $R_i/L_{i}$ is prosoluble (possibly trivial).

 As a special case of a general result in Wilson's paper \cite{wil83} we have the following.

\begin{lemma}[{see \cite[Lemma~2]{wil83}}]\label{l-nsl}
If, for some positive integer $m$, all continuous finite quotients of a profinite group~$G$ have nonsoluble length at most $m$, then $G$ has finite nonprosoluble length at most~$m$.
\end{lemma}

We are now ready to prove the key proposition.

\begin{proposition}\label{pr-fnl}
Suppose that $G$ is a profinite group
in which the word $[x,{}_ny]$ has strictly less than $2^{\aleph_0}$ values. Then $G$ has finite nonprosoluble length.
\end{proposition}

\begin{proof}
Let $ H=\bigcap G^{(i)} $
be the intersection of the derived series of $G$.
Then $H=[H,H]$. Indeed,
if $H\ne [H,H]$, then the quotient $G/[H,H]$ is a prosoluble group by Lemma~\ref{l-prosol-by-prosol}, whence $\bigcap G^{(i)}=H\leq [H,H]$, a contradiction.
Since the quotient $G/H$ is prosoluble, it is sufficient to prove the proposition for $H$. Thus, we can assume from the outset that $G=[G,G]$.

Let $T$ be a Sylow $2$-subgroup of $G$. By Theorem~\ref{t2} the group $T$ is locally nilpotent. Consider the subsets of the direct product $T\times T$
$$
S_{i}=\{(x,y)\in T\times T\mid \text{the subgroup }\langle x,y\rangle\text{ is nilpotent of class at most }i\}.
$$
Note that each subset $S_{i}$ is closed in the product topology of $T\times T$, because the condition defining $S_i$ means that all commutators of weight $i+1$ in $x,y$ are trivial. Since every $2$-generator subgroup of $T$ is nilpotent, we have
$$
\bigcup _iS_{i}=T\times T.
$$
By
Theorem~\ref{bct} one of the sets $S_i$ contains an open subset of $T\times T$. This means that there are cosets $aN$ and $bN$ of an open normal subgroup $N$ of $T$ and a positive integer $c$  such that
\begin{equation}\label{e-2nilp}
\langle x,y\rangle\text{ is nilpotent of class }c\text{ for any }x\in aN,\; y\in bN.
\end{equation}

Let $K$ be an open normal subgroup of $G$ such that $K\cap T\leq N$. If we replace $N$ by $K\cap T$, then \eqref{e-2nilp} still holds with the same $a,b$. Hence we can assume that $N$ is a Sylow $2$-subgroup of $K$.

By \cite[Lemma 2.8.15]{rib-zal} there is a subgroup $H\leq G$ such that
$G=KH$ and $K\cap H$ is pronilpotent.
Since $H$ is virtually pronilpotent and every element has a countable Engel sink, by Corollary~\ref{c-virt} the subgroup $\gamma _\infty(H)$ is finite. Recalling our assumption that $G=[G,G]$, we obtain
$$
G=[G,G]=\gamma _{\infty}(G)\leq \gamma _{\infty}(HK)\leq \gamma _{\infty}(H)K.
$$
Thus, $G=\gamma _{\infty}(H)K$, where $\gamma _{\infty}(H)$ is a finite subgroup.

Hence we can choose the coset representative $a$ satisfying \eqref{e-2nilp}
in a conjugate of a Sylow $2$-subgroup of $\gamma _{\infty}(H)$, and therefore having finite order, say, $|a|=2^m$.

For any $y\in bN$ the $2$-subgroup $\langle a,y\rangle$ is nilpotent of class at most $c$, while  $a^{2^m}=1$. Then
\begin{equation}\label{e-ay2nc}
[a,y^{2^{m(c-1)}}]=1.
 \end{equation}
 This follows from well-known commutator formulae (and for any $p$-group); see, for example, \cite[Lemma~4.1]{shu00}.

In particular,  for any $z\in N$ by using \eqref{e-ay2nc} we obtain
\begin{equation}\label{e-2eng2}
[z,\,{}_c y^{2^{m(c-1)}}]=[az,\,{}_c y^{2^{m(c-1)}}]=1,
\end{equation}
since $\langle az,y^{2^{m(c-1)}}\rangle$ is  a subgroup of  $\langle az,y\rangle$, which is nilpotent of class $c$  by \eqref{e-2nilp}.

Our aim is to show that there is a uniform bound, in terms of $|G:K|$, $c$, and $m$,  for the nonsoluble length of all finite quotients of $G$ by open normal subgroups.
Let $M$ be an open normal subgroup of $G$ and let the bar denote the images in $\bar G=G/M$. It is clearly sufficient to obtain a required bound for the nonsoluble length of $\bar K$.

Let $R_0$ be the soluble radical of $\bar K$, and $L_1$ the inverse image of the generalized Fitting subgroup of $\bar K/R_0$, so that
\begin{equation}\label{e-soc}
L_1/R_0=S_1\times S_2\times \dots\times S_k
\end{equation}
is a direct product of non-abelian finite simple groups. Note that $R_0$ and $L_1$ are normal subgroups of $\bar G$. The group $\bar G$ acting by conjugation induces a permutational action on the set $\{S_1,S_2,\dots ,S_k\}$. The kernel of the restriction of this permutational action to $\bar K$ is contained in the inverse image $R_1$ of the soluble radical of $\bar K/L_1$:
\begin{equation}\label{e-soc2}
\bigcap _iN_{\bar K}(S_i)\leq R_1.
\end{equation}
This follows from the validity of Schreier's conjecture on the solubility of the outer automorphism groups of non-abelian finite simple groups, confirmed by the classification of the latter, because $L_1/R_0$ contains its centralizer in $\bar K/R_0 $.

Let $e$ be the least positive integer such that $2^{e}\geq c$, and let $t= 2^{m(c-1)+e}$. We claim that for any   $y\in \bar b\bar N$ the element  $y^{2^t}$ normalizes each factor $S_i$ in \eqref{e-soc}.
Arguing by contradiction, suppose that the element $y^{2^t}$ has a nontrivial orbit on the set of the  $S_i$. Then the element $y^{2^{m(c-1)}}$ has an orbit of length $2^s\geq  2^{e+1}$ on this set; let $\{T_1,T_2,\dots , T_{2^s}\}$ be such an orbit cyclically permuted by $y^{2^{m(c-1)}}$.
Since non-abelian finite simple groups have even order (by the Feit--Thompson theorem \cite{fei-tho}) and the subgroups $S_i$ are subnormal in $\bar K/R_0$, each subgroup $S_i$ contains a nontrivial element of $\bar NR_0/R_0$. If $x$ is a nontrivial element of $T_1\cap  \bar NR_0/R_0$, then the commutator
$$
[x,\,{}_c \bar y^{2^{m(c-1)}}],
$$
written as an element of $T_1\times T_2\times \dots\times  T_{2^s}$, has a nontrivial component in $T_{c+1}$ since $2^s\geq  2^{e+1}>c$. This, however, contradicts~\eqref{e-2eng2}.

Thus, for any element  $y\in \bar b\bar N$ the power $y^{2^t}$ normalizes each factor $S_i$ in \eqref{e-soc}. Let $2^{d}$ be the highest power of $2$ dividing $|G:K|$, and let $u=\max\{t,d\}$. Then $y^{2^u}\in R_1$ by \eqref{e-soc2}, since $y^{2^u}\in \bar K$ and $y^{2^u}$ normalizes each $S_i$ in \eqref{e-soc} by the choice of $u$.

As a result, in the quotient $\bar G/R_1$ all elements of the coset $\bar b\bar NR_1/R_1$ of the Sylow $2$-subgroup $\bar NR_1/R_1$ of $\bar K/R_1$ have exponent dividing $2^u$. We can now apply  Theorem~\ref{t-wil83}, by which the nonsoluble length of $\bar K/R_1$ is at most $u$. Then the  nonsoluble length of $\bar K$ is at most $u+1$. Clearly, the nonsoluble length of $\bar G/\bar K$ is bounded in terms of  $|G:K|$.
As a result, since the number $u$ depends only on $|G:K|$, $m$, and $c$, the nonsoluble length of $\bar G$ is bounded in terms of these parameters only. Since this  holds for any quotient of the profinite group $G$ by a normal open subgroup,  the group $G$ has finite nonprosoluble length by Lemma~\ref{l-nsl}. This completes the proof of Proposition~\ref{pr-fnl}.
\end{proof}

We are now ready to handle the general case of profinite groups using Corollary~\ref{c-virt} on virtually prosoluble groups and induction on the nonprosoluble length.
First we eliminate infinite Cartesian products of non-abelian finite simple groups.

\begin{lemma}\label{l-cart}
Suppose that $G$ is a profinite group that is a Cartesian product of non-abelian finite simple groups. If the word $[x,{}_ny]$ has strictly less than $2^{\aleph_0}$ values in $G$, then $G$ is finite.
\end{lemma}

\begin{proof}
Suppose the opposite: then $G$ is a Cartesian product of infinitely many non-abelian finite simple groups $G_i$ over an infinite set of indices $i\in I$. Every subgroup $G_i$ contains an element $g_i\in G_i$ with a nontrivial smallest Engel sink $\mathscr E(g_i)\ne \{1\}$. (Actually, any nontrivial element of $G_i$ has a nontrivial Engel sink, since an Engel element of a finite group belongs to its Fitting subgroup by Baer's theorem \cite[Satz~III.6.15]{hup}.) By Lemma~\ref{l-min}, for any $s\in \mathscr E(g_i)$ we have $s=[s,{}_{k}g_i]$ for some $k\in {\mathbb N}$, and then also
\begin{equation}\label{e-cycl2}
s=[s,{}_{kl}g_i]\quad \text{for any}\;\, l\in {\mathbb N}.
\end{equation}

For every $i$, choose a nontrivial element $s_{i}\in \mathscr E(g_i)\subseteq G_i$.
For any subset $J\subseteq I$, consider the (infinite) product
$$
s_{J}=\prod _{j\in J} s_{j}.
$$
Let $$
g=\prod _{i\in I} g_{i}.
$$
If  $\mathscr E(g)$ is any Engel sink of $g$ in $G$, then for some $k\in {\mathbb N}$ the commutator $[s_{J},{}_kg]$ belongs to $\mathscr E(g)$. Because of the properties \eqref{e-cycl2}, all the components of $[s_{J},{}_kg]$ in the factors $G_j$ for $j\in J$ are nontrivial, while all the other components in $G_i$ for $i\not\in J$ are trivial by construction. Therefore for different subsets $J\subseteq I$ we thus obtain different elements of  $\mathscr E(g)$. The infinite set $I$ has at least $2^{\aleph_0}$ different subsets, whence $\mathscr E(g)$ has cardinality at least $2^{\aleph_0}$. But we can choose
$$\mathscr E(a)=\bigcup_{i\in \N}\{[g,{}_{n+i}a]\mid g\in G\},$$
which is a countable union of sets each having cardinality less than  $2^{\aleph_0}$ by hypothesis. This Engel sink $\mathscr E(a)$ therefore also has cardinality less than  $2^{\aleph_0}$, a contradiction. 
\end{proof}

We now finish the proof of Theorem~\ref{t1}, which also completes the proof of the main Theorem~\ref{t-main}, as explained in the introduction.

\begin{proof}[Proof of Theorem~\ref{t1}] Recall that $G$ is a profinite group in which the word $[x,{}_ny]$ has strictly less than $2^{\aleph_0}$ values. We need to show that $G$ has a finite normal subgroup $N$ such that $G/N$ is locally nilpotent.

By Proposition~\ref{pr-fnl} the group $G$ has finite nonprosoluble length $l$. This means that $G$ has a normal series
\begin{equation*}
1=L_0\leqslant R_0 <  L_1\leqslant R_1< L_1  \leqslant \dots \leqslant R_{l}=G
\end{equation*}
in which each quotient $L_i/R_{i-1}$ is a (nontrivial) Cartesian product of non-abelian finite simple groups, and each quotient $R_i/L_{i}$ is prosoluble (possibly trivial).
We argue by induction on $l$. When $l=0$, the group $G$ is prosoluble, and the result follows by Theorem~\ref{t3}.

Now let $l\geq 1$. By Lemma~\ref{l-cart} each of the nonprosoluble factors $L_i/R_{i-1}$ is finite. In particular, the subgroup $L_1$ is virtually prosoluble, and therefore $\gamma _{\infty}(L_1)$ is finite by Corollary~\ref{c-virt}. The quotient $R_1/ \gamma _{\infty}(L_1)$ is prosoluble by Lemma~\ref{l-prosol-by-prosol}. Hence the  nonprosoluble length of $G/\gamma _{\infty}(L_1)$ is $l-1$. By the induction hypothesis we obtain  that $\gamma _{\infty}(G/\gamma _{\infty}(L_1))$ is finite, and therefore $\gamma _{\infty}(G)$ is finite. By Theorem~\ref{t2} the quotient $G/\gamma _{\infty}(G)$ is locally nilpotent, and the proof is complete.
\end{proof}

 \section*{Acknowledgements}
 The second author was supported by FAPDF and CNPq-Brazil.

\end{document}